\documentclass{amsart}
\usepackage{amsmath,amsfonts,latexsym,amssymb}

\newtheorem{theorem}{Theorem}
\newtheorem{lemma}{Lemma}
\newtheorem{corollary}{Corollary}

\theoremstyle{remark}
\newtheorem{remark}{Remark}
\newtheorem*{ack}{Acknowledgements}

\newcommand{\Sym}{\text{Sym}^2}

\begin{document}

\markboth{Ritabrata Munshi}{Bounds for twisted symmetric square $L$-functions}
\title[Bounds for twisted symmetric square $L$-functions]{Bounds for twisted symmetric square $L$-functions}

\author{Ritabrata Munshi}   
\address{School of Mathematics, Tata Institute of Fundamental Research, 1 Homi Bhabha Road, Colaba, Mumbai 400005, India.}     
\email{rmunshi@math.tifr.res.in}

\begin{abstract}
Let $f\in S_k(N,\psi)$ be a newform, and let $\chi$ be a primitive character of conductor $q^{\ell}$. Assume that $q$ is a prime and $\ell>1$. In this paper we describe a method to establish convexity breaking bounds of the form
$$
L\left(\tfrac{1}{2},\Sym f\otimes\chi\right)\ll_{f,\varepsilon} q^{\frac{3}{4}\ell-\delta_{\ell}+\varepsilon}
$$
for some $\delta_{\ell}>0$ and any $\varepsilon>0$. In particular, for $\ell=3$ we show that the bound holds with $\delta_{\ell}=\frac{1}{4}$. 
\end{abstract}


\maketitle

%\tableofcontents
%======================================================================================================================
%======================================================================================================================

\section{Introduction}
\label{intro}

Let $f\in S_{k_f}(N_f,\psi_f)$ be a newform of weight $k_f$, level $N_f$ and nebentypus $\psi_f$. Let $\chi$ be a primitive character of conductor $M_{\chi}$, with $(M_{\chi},N_f)=1$. The `conductor' of the degree three $L$-function $L(s,\Sym f \otimes \chi)$ is given by $N_fM_{\chi}^3$. Accordingly, from the functional equation and the convexity principle we get the bound
$$
L\left(\tfrac{1}{2},\Sym f\otimes\chi\right)\ll_{f,\varepsilon} M_{\chi}^{\frac{3}{4}+\varepsilon}
$$
for the central value. Getting a subconvex bound, i.e. getting a bound of the form $M_{\chi}^{\theta}$ with $\theta<3/4$, in this context is an intriguing problem in the analytic theory $L$-functions. The Generalized Lindel\"of Hypothesis, which is a consequence of the Generalized Riemann Hypothesis, predicts that the exponent can be taken to be any positive real number however small.\\

In this paper we will describe a method to prove a subconvex bound for the twisted $L$-function in the case where the conductor $M_{\chi}=q^{\ell}$, with $q$ a prime number and $\ell>1$. 
\begin{theorem}
\label{mthm}
Let $f\in S_{k_f}(N_f,\psi_f)$ be a newform of weight $k_f$, level $N_f$ and nebentypus $\psi_f$. Let $\chi$ be a character of conductor $M_{\chi}=q^{\ell}$ where $q$ is a prime number and $\ell>1$. Then we have
$$
L\left(\tfrac{1}{2},\text{$\rm{Sym}^2$} f\otimes\chi\right)\ll_{f,\ell,\varepsilon} q^{\frac{3}{4}\ell-\delta_{\ell}+\varepsilon}
$$
for some $\delta_{\ell}>0$, which depends only on $\ell$. The implied constant depends on $f$, $\ell$ and $\varepsilon$, but does not depend on $q$. 
\end{theorem}
This result is the first instance of a subconvexity bound in the level aspect for a genuine degree three $L$-function which is not self-dual. This nicely complements the recent work of Blomer \cite{Bl2} who has proved
$$
L\left(\tfrac{1}{2},\Sym f\otimes\chi\right)\ll_{f,\varepsilon} M_{\chi}^{\frac{5}{8}+\varepsilon},
$$
in the case of quadratic characters $\chi$ and $f$ of full level. As in \cite{Bl2}, the situation considered here is quite special and it still remains a major open problem to prove subconvexity of a general degree three $L$-function in the level aspect.\\

The only other subconvexity result known in the case of the symmetric square $L$-function is in the $t$-aspect, which has been established by X. Li \cite{L}. The method used by Blomer, or Li, is very much different from our method in this paper. In particular both Blomer and Li employ $GL(3)$ Voronoi summation formula, while we only require $GL(2)$ techniques. Another crucial input in their work is a deep result of Lapid \cite{La} on the positivity of the central value. Since, in this paper, we estimate the second moment, we do not have to rely on positivity. \\

The assumption that $f$ is a holomorphic form in Theorem \ref{mthm} is made only for simplicity, and a similar subconvex bound can also be proved for Maass forms. In that case we have to use Kuznetsov formula instead of the Petersson formula in Section \ref{rec-poi1}. After that most of the calculations remain unchanged, except in Section \ref{integral} where we need to make some alterations. Also the assumption that $q$ is a prime is made only for technical reasons. In fact, the technique works for a large class of composite conductors, including square-free ones having more than one prime factors satisfying certain relative size restrictions (see \cite{Mu}). The prime-power conductors, $q^{\ell}$ with $\ell>1$, form a distinctive subset in this class, and for notational and technical simplicity we restrict ourselves to this subclass.  Finally, it should be clear that  the result also holds for other points on the critical line $\frac{1}{2}+it$, and in that case the constant also depends on $t$, but grows at most polynomially with $t$. \\  
 
The proof of the theorem is based on the method of moments. But the choice of the family is a little strange. It turns out that the natural family consisting of all the twists of conductor $q^{\ell}$ is not a good choice. This, in fact, leads to the difficult shifted convolution sum 
$$
\mathop{\sum\sum}_{\substack{n,m\sim q^{3\ell/2}\\n\equiv m \bmod q^{\ell}}}\lambda_f(n^2)\lambda_f(m^2).
$$
Neither the circle method nor the $GL(2)$ spectral theory seems to be effective to deal with this problem directly. \\

Now we will briefly outline our approach. First using approximate functional equation we get that the twisted value $L\left(\tfrac{1}{2},\text{$\rm{Sym}^2$} f\otimes\chi\right)$ is given by a rapidly converging series of effective length $q^{\frac{3\ell}{2}}$. Trivial estimation of this series, ignoring the oscillation in the sign of the Fourier coefficients and the character values, leads back to the convexity bound. So to break the convexity barrier one has to utilize the oscillation to produce cancellation in the sum. At this point one can cut the sum into dyadic segments and try to get bounds for each segments separately. It follows that
\begin{align}
\label{one}
L\left(\tfrac{1}{2},\text{$\rm{Sym}^2$} f\otimes\chi\right)\ll_{A,\varepsilon} q^{\varepsilon}\sum_{N}\frac{|L_f(N)|}{\sqrt{N}}\left(1+\frac{N}{q^{3\ell/2}}\right)^{-A}
\end{align}
for any $A>0$. Here $N$ ranges over $2^{\alpha}$ with $\alpha>-1/2$, and $L_f(N)$ is a linear form given by
\begin{align}
\label{lf}
L_f(N):=\sum_{n}\lambda_f(n^2)\chi(n)h(n/N),
\end{align}
where $h$ is a smooth weight function supported in $[1,2]$. For smaller segments the trivial bound $L_f(N)\ll N^{1+\varepsilon}$ is good enough. So the problem boils down to proving a nontrivial bound for the linear form in the range $N\asymp q^{\frac{3\ell}{2}}$, i.e. square-root of the conductor. \\

As hinted by our notation in \eqref{lf}, we will keep $\chi$ fixed, and seek for an appropriate family containing the given form $f$. In fact, we consider $f$ as a Hecke oldform in the larger space $S_k(M)$ of level $M$. Since we are not looking for optimal bounds, we will take 
$$
M=\begin{cases}16N_fq^{3j+1} & \text{for $\ell=2j+1$ odd},\\16N_fq^{3j-1} & \text{for $\ell=2j>2$ even},\\ 16N_fq^3 &\text{for $\ell=2$.} 
\end{cases}
$$
Then we compute the second moment of the associated linear form
$$
\sum_{g\in \mathcal B_k(M)}\omega_g^{-1}\left|\mathcal A_g\right|^2\left|L_g(N)\right|^2,
$$ 
where $\mathcal B_k(M)$ stands for an orthogonal basis of the space containing the given form $f$, and 
$$
\omega_g=\|g\|_M^2\frac{(4\pi)^{k-1}}{\Gamma(k-1)}
$$ 
denotes the spectral weight. The amplifier $\mathcal A_g$ is required to take care of the diagonal contribution only in the case $\ell=2$. In this case we further need the forms in the basis $\mathcal B_k(M)$ to be Hecke forms. \\

The value of $\delta_{\ell}$ can be computed explicitly, though we do not try to do this here for all $\ell$. In the next section we give some arguments to show that with the above choice for $M$, we have $\delta_{\ell}$ to be at least $\frac{1}{4}$ for all odd $\ell>1$. It follows, as one expects, that with some fine tuning (i.e. by choosing $M$ appropriately) we can have $\delta_{\ell}$ to increase linearly with $\ell$. \\

To maintain notational simplicity and to keep the arguments as transparent as possible, we will give complete details only for the case of full level and $\ell=3$. It should be clear that the arguments generalize to the case of general level and general nebentypus. The major part of this paper goes into proving the following result.
\begin{theorem}
\label{thm2}
Let $L_g$ be the linear form as defined in \eqref{lf}. For $N\leq q^{\frac{9}{2}+\varepsilon}$ we have
$$
\sum_{g\in \mathcal B_k(16q^4)}\omega_g^{-1}\left|L_g(N)\right|^2\ll_{k,\varepsilon} Nq^{\varepsilon}.
$$
\end{theorem} 
Let $f$ be a holomorphic Hecke form of full level. Let $\mathcal B_k(16q^4)$ be an orthogonal basis of $S_k(16q^4)$ containing the form $f$, as described above. The next result follows immediately from Theorem \ref{thm2}, by using positivity to drop all the terms from the sum except the term corresponding to the given form $f$. Note that 
$$
\|f\|_{16q^4}^2=[\Gamma_0(1):\Gamma_0(16q^4)]\frac{2}{\pi}\frac{\Gamma(k)}{(4\pi)^k}L\left(1,\Sym f\right)\ll_k q^4.
$$
\begin{corollary}
\label{cor-thm2}
For $N\leq q^{\frac{9}{2}+\varepsilon}$ we have
$$
L_f(N)\ll _{k,\varepsilon} \sqrt{N}q^{2+\varepsilon}.
$$
\end{corollary}
Observe that the above bound is better than the trivial bound $L_f(N)\leq N^{1+\varepsilon}$ in the case when $N>q^4$. In \eqref{one} we use the trivial bound in the range $N\leq q^4$, and in the complementary range $N>q^4$ we substitute the bound from Corollary \ref{cor-thm2}. 
\begin{corollary}
\label{cor-thm3}
Let $f\in S_k(1)$ be a form of weight $k$ and full level. Suppose that the conductor of $\chi$ is $M_{\chi}=q^3$ where $q$ is a prime number. Then we have
$$
L\left(\tfrac{1}{2},\text{$\rm{Sym}^2$} f\otimes\chi\right)\ll_{k,\varepsilon} q^{2+\varepsilon}.
$$
\end{corollary}
This proves Theorem \ref{mthm} in the case $\ell=3$ with $\delta_{\ell}=1/4$. Recall that in this case the convexity bound is given by $q^{\frac{9}{4}+\varepsilon}$.

\ack
The author wishes to thank Henryk Iwaniec for introducing him to this subject. The author also thanks Valentin Blomer, Matthew Young and the referee for their helpful comments. A part of this work was done while the author was visiting the Indian Statistical Institute, Bangalore.  

%====================================================================================================================
%====================================================================================================================

\section{Sketch of the proof}

In this section we make some brief remarks about the main ingredients of the proof. The main layout of the proof is quite simple, but major complications arise due to coprimality issues and the oscillation in the Bessel function outside the transition range. In the rest of the paper we make the following argument rigorous and complete in the case $\ell=3$.  \\
  
Let $\ell=2j+1$. Here the worst case scenario corresponds to getting a nontrivial bound for $L_f(N)$ with $N=q^{\frac{3}{2}(2j+1)}$. Our target is to get a bound for the average
$$
\sum_{g\in \mathcal B_k(q^{3j+1})}\omega_g^{-1}\left|\sum_{n}\lambda_f(n^2)\chi(n)h\left(\frac{n}{q^{\frac{3}{2}(2j+1)}}\right)\right|^2.
$$ 
We open the absolute square, interchange the order of summation and then apply the Petersson formula to the sum over $g$. The sum splits as
$$
\text{diagonal}+\text{off-diagonal}.
$$
The diagonal is easily seen to be bounded by $q^{\frac{3}{2}(2j+1)}$, which reflects the expected size of the average, and is satisfactory for our purpose. The main problem boils down to proving a strong bound for the off-diagonal contribution which is roughly given by
$$
\sum_{c\sim q^{3j+2}}\frac{1}{q^{3j+1}c}\mathop{\sum\sum}_{n, m\sim q^{\frac{3}{2}(2j+1)}}\chi(n)\overline{\chi(m)}S(n^2,m^2;q^{3j+1}c).
$$
Note that here we are just focusing at the transition range of the Bessel function, and we are ignoring the weight functions. Now we want to apply the Poisson summation formula to the sum over $n$ and $m$. The modulus of the sum is $q^{3j+1}c$ which is of size $q^{3(2j+1)}$, the square of the length of the sums over $n$ and $m$. As such Poisson at this point does not reduce the length of the sum. However Poisson in one of the variables give certain advantage in the structure, as the Kloosterman sum dissolves to give rise to a Gauss-Ramanujan type sum. But, as we have seen in many situations related to the symmetric square case (see e.g. \cite{IM}), this is not that helpful and we end up with an usual deadlock situation. \\

To resolve this issue (in Section \ref{rec-poi1}) we use the fact that a large part of the Kloosterman sum can be evaluated explicitly, which is why we started with an average over such a large family. Indeed it turns out that the off-diagonal is essentially given by 
$$
\sum_{c\sim q^{3j+2}}\frac{1}{q^{\frac{1}{2}(3j+1)}c}\mathop{\sum\sum}_{n, m\sim q^{\frac{3}{2}(2j+1)}}\chi(n)\overline{\chi(m)}\left(\frac{cnm}{q}\right)^{j+1}e\left(\frac{2\overline{c}nm}{q^{3j+1}}\right)S(\overline{q^{3j+1}}n^2,\overline{q^{3j+1}}m^2;c).
$$ 
(The quadratic character appears when $j$ is even, i.e. when $3j+1$ is odd.) Now we can use reciprocity to get
$$
\sum_{c\sim q^{3j+2}}\frac{1}{q^{\frac{1}{2}(3j+1)}c}\left(\frac{c}{q}\right)^{j+1}\mathop{\sum\sum}_{n, m\sim q^{\frac{3}{2}(2j+1)}}\chi'(n)\overline{\chi'(m)}e\left(\frac{-2\overline{q^{3j+1}}nm}{c}\right)S(\overline{q^{3j+1}}n^2,\overline{q^{3j+1}}m^2;c),
$$ 
which reduces the modulus from $q^{3j+1}c$ to $q^{2j+1}c$. Here $\chi'(.)=\chi(.)(\frac{.}{q})^{j+1}$. (Notice that $nm\asymp q^{3j+1}c\asymp q^{3(2j+1)}$, so that $e(2nm/q^{3j+1}c)$ is essentially `flat' and can be absorbed in the weight function.) Now we apply the Poisson summation formula on both the sums over $n$ and $m$. The dual sums have length $q^{2j+1}c/q^{\frac{3}{2}(2j+1)}\asymp q^{2j+\frac{3}{2}}$. So there is a reduction in the length. Moreover we also gain structural advantage. The sum at this point is roughly given by
$$
\frac{1}{q^{2j+\frac{3}{2}}}\sum_{c\sim q^{3j+2}}\mathop{\sum\sum}_{n, m\sim q^{2j+\frac{3}{2}}}\chi'(n)\overline{\chi'(m)}\left(\frac{c}{nm}\right)e\left(\frac{\overline{q^{j+1}}nm}{c}\right).
$$
(The explicit evaluation of the character sum which is required for this reduction, is carried out in Section \ref{sec-char}.) Observe that the Kloosterman sum has vanished which is the usual advantage of applying Poisson. Also note that we have achieved a saving of $q^j$, due to the reduction in the length. Next we want to apply the Poisson summation on the sum over $c$. For this (in Section \ref{sec-poi}) we again apply reciprocity to get
$$
\frac{1}{q^{2j+\frac{3}{2}}}\mathop{\sum\sum}_{n, m\sim q^{2j+\frac{3}{2}}}\chi'(n)\overline{\chi'(m)}\sum_{c\sim q^{3j+2}}\left(\frac{c}{nm}\right)e\left(\frac{-\overline{c}nm}{q^{j+1}}\right).
$$
The modulus for the sum over $c$ is given by $nmq^{j+1}$ which is of the size $q^{5j+4}$. Hence after Poisson the dual sum is of length $q^{2j+2}$. In fact we arrive at the following expression
$$
\frac{1}{q^{2j+2}}\mathop{\sum\sum}_{n, m\sim q^{2j+\frac{3}{2}}}\chi'(n)\overline{\chi'(m)}\sum_{c\sim q^{2j+2}}\left(\frac{nm}{c}\right)S(nm,\overline{nm}c;q^{j+1}).
$$
(The explicit evaluation of the associated character sum is carried out in Section \ref{sec-char2}.) We have saved another $\sqrt{q^j}$ due to the reduction in the length of the sum over $c$. But more importantly in the Kloosterman sum we can make a change of variables to make it free from $n$ and $m$. This crucial separation of variable is the key. Interchanging the order of summation we get
$$
\frac{1}{q^{2j+2}}\sum_{c\sim q^{2j+2}}S(1,c;q^{j+1})\sum_{n\sim q^{2j+\frac{3}{2}}}\chi'(n)\left(\frac{n}{c}\right)\sum_{m\sim q^{2j+\frac{3}{2}}}\overline{\chi'(m)}\left(\frac{m}{c}\right).
$$
Now we apply Cauchy and use the Weil bound for the Kloosterman sum. With this we arrive at 
$$
\frac{1}{q^{\frac{3}{2}(j+1)}}\sum_{c\sim q^{2j+2}}\left|\sum_{n\sim q^{2j+\frac{3}{2}}}\chi'(n)\left(\frac{n}{c}\right)\right|^2.
$$
Our last step (in Section \ref{sec-large}) is an application of the large sieve inequality for quadratic characters (see \cite{HB}), which shows that the above expression is dominated by  
$$
\frac{1}{q^{\frac{3}{2}(j+1)}}(q^{2j+2}+q^{2j+\frac{3}{2}})q^{2j+\frac{3}{2}}\ll q^{\frac{5}{2}j+2}.
$$
This gives an upper bound for the contribution from the off-diagonal. Clearly this is more than satisfactory for our purpose. The bound coincides with the size of the diagonal in the case $j=1$, otherwise the off-diagonal is smaller than the diagonal. This shows that our choice of the level is optimal in the case $j=1$, but can be improved otherwise.\\

Observe that we are losing a $\sqrt{q}$ in the last estimate due to the difference in the lengths of the sums over $c$ and $n$, $m$. In the case $\ell=2$, we do not lose this extra $\sqrt{q}$, and in fact the off-diagonal contribution can be shown to be of smaller magnitude compared to that of the diagonal. This is important as in this case we need to use an amplifier and when we introduce an amplifier we gain in the diagonal but lose in the off-diagonal.

%+++++++++++++++++++++++++++++++++++++++++++++++++++++++++++++++++++++++++++++++++
%+++++++++++++++++++++++++++++++++++++++++++++++++++++++++++++++++++++++++++++++++

\section{Preliminaries}
\label{prelim}

In this section we will briefly recall some fundamental facts about holomorphic forms and their $L$-functions (for details see \cite{IK}). Let $f\in S_k(1)$ be a newform with Fourier expansion 
$$
f(z)=\sum_{n=1}^{\infty}\lambda_f(n)n^{\frac{k-1}{2}}e(nz).
$$ 
For $s=\sigma+it$ with $\sigma>1$, the associated $L$-function is given by 
$$
L(s,f)=\sum_{n=1}^{\infty}\frac{\lambda_f(n)}{n^s}=\prod_p \left(1-\frac{\alpha_f(p)}{p^s}\right)^{-1} \left(1-\frac{\beta_f(p)}{p^s}\right)^{-1}.
$$
The local parameter $\alpha_f(p)$ and $\beta_f(p)$ are related to the normalized Fourier coefficients in the following way
$$
\alpha_f(p)+\beta_f(p)=\lambda_f(p),\;\;\;\;\alpha_f(p)\beta_f(p)=1
$$
Now let $\chi$ be a primitive Dirichlet character of modulus $M_{\chi}$. Then we define the twisted symmetric square $L$-function by the degree three Euler product 
$$
L(s,\Sym f\otimes\chi)=\prod_p \left(1-\frac{\alpha_f^2(p)\chi(p)}{p^s}\right)^{-1} \left(1-\frac{\chi(p)}{p^s}\right)^{-1} \left(1-\frac{\beta_f^2(p)\chi(p)}{p^s}\right)^{-1},
$$ 
for $\sigma>1$. In this half-plane we have
$$
L(s,\Sym f\otimes\chi)=L(2s,\chi^2)\sum_{n=1}^{\infty}\frac{\lambda_f(n^2)\chi(n)}{n^s}.
$$ \\

It is well-known that this $L$-function extends to an entire function and satisfies a functional equation (see \cite{Li}). Indeed we have a completed $L$-function defined as
$$
\Lambda(s,\Sym f\otimes\chi)=M_{\chi}^{3s/2}\gamma(s)L(s,\Sym f\otimes\chi)
$$
where $\gamma(s)$ is essentially a product of three gamma functions $\Gamma(\frac{s+\kappa_j}{2})$, $j=1,2,3$, with $\kappa_j$ depending on the weight of $f$ and the parity of the character $\chi$, such that the functional equation is given by
$$
\Lambda(s,\Sym f\otimes\chi)=\varepsilon(f,\chi)\Lambda(1-s,\Sym f\otimes\chi).
$$
Here $\text{Re}(\kappa_j)>0$ and the $\varepsilon$-factor satisfies $|\varepsilon(f,\chi)|= 1$. Using standard arguments we get that the twisted $L$-value $L(\frac{1}{2},\Sym f\otimes\chi)$, is given by the rapidly converging series (the approximate functional equation)
\begin{align}
\label{afe}
\sum_{n=1}^{\infty}\frac{\lambda_f(n^2)\chi(n)}{\sqrt{n}}V\left(\frac{n}{M_{\chi}^{3/2}}\right)+
\varepsilon(f,\chi)\sum_{n=1}^{\infty}\frac{\lambda_f(n^2)\overline{\chi(n)}}{\sqrt{n}}
V\left(\frac{n}{M_{\chi}^{3/2}}\right),
\end{align}
where
$$
V(y)=\frac{1}{2\pi i}\int_{(3)}
\frac{\gamma\left(\frac{1}{2}+u\right)}{\gamma\left(\frac{1}{2}\right)}
\left(\cos \frac{\pi u}{4A}\right)^{-12A}L(1+2u,\chi^2)y^{-u}\frac{du}{u}.
$$ 
(Here $A$ is a sufficiently large positive integer.) The weight function $V(y)$ satisfies the bound 
$$
y^jV^{(j)}(y)\ll_{j,A} y^{-A}.
$$
Breaking the sum in \eqref{afe} into dyadic blocks, it follows that (see e.g. \cite{IM})
$$
L(\tfrac{1}{2},\Sym f\otimes\chi)\ll_{A,\varepsilon} q^{\varepsilon}\sum_{N}\frac{\left|L_f(N)\right|}{\sqrt{N}}\left(1+\frac{N}{M_{\chi}^{3/2}}\right)^{-A}
$$
where $N$ ranges over the values $2^{\alpha}$ with $-1/2\leq \alpha$, and 
$$
L_f(N)=\sum_{n}\lambda_f(n^2)\chi(n)h(n/N).
$$
Here $h(.)$ is a smooth function supported in $[1,2]$. For any $\varepsilon>0$ we can choose $A$ appropriately so that the contribution from $N>M_{\chi}^{3/2+\varepsilon}$ is negligible. Also for the smaller values of $N$ we can estimate the sum trivially. It turns out that the worst case scenario corresponds to the case where $N\asymp M_{\chi}^{3/2}$. \\

Now $f$ can be considered as a Hecke form in the larger space $S_k(M)$ of cusp forms of level $M$. Then we select an orthogonal basis $\mathcal B_k(M)$ of the space $S_k(M)$ containing the form $f$. For any form $g\in S_k(M)$ we have the Fourier expansion $g(z)=\sum \lambda_g(n)n^{\frac{k-1}{2}}e(nz)$. Let 
$$
\left<g_1,g_2\right>_M=\int_{\Gamma_0(M)\backslash\mathbb H}g_1(z)\overline{g_2(z)}y^{k-2}dxdy
$$ 
denote the Petersson inner product at level $M$. Let $\|g\|_M^2=\left<g,g\right>_M$ denote the Petersson norm at level $M$. Then we have the Petersson formula
\begin{align}
\label{pform}
\frac{\Gamma(k-1)}{(4\pi)^{k-1}}\sum_{g\in \mathcal B_k(M)}\frac{\lambda_g(n)\lambda_g(m)}{\|g\|_M^2}=\delta(n,m)+
2\pi i^{-k}\sum_{c=1}^{\infty}\frac{S(n,m;cM)}{cM}J_{k-1}\left(\frac{4\pi \sqrt{nm}}{cM}\right),
\end{align}
where $S(n,m;cM)$ denotes the Kloosterman sum and $J_{k-1}(.)$ is the $J$-Bessel function of order $k-1$.  \\

It is well known (see \cite{IM} or \cite{GR}) that the Bessel function can be expressed as
\begin{align}
\label{bessel-split}
J_{k-1}(2\pi x)=e(x){W}_k(x)+e(-x)\bar W_k(x)
\end{align}
where $W_k:(0,\infty)\rightarrow \mathbb C$ is a smooth function satisfying the bound
\begin{align}
\label{bessel-bd0}
x^jW_k^{(j)}(x)\ll \max\{x^{k-1},x^{-\frac{1}{2}}\}. 
\end{align}

%========================================================================================================================
%========================================================================================================================
%========================================================================================================================
%========================================================================================================================

\section{Reciprocity and Poisson summation - I}
\label{rec-poi1}

Let $f\in S_k(1)$ be a Hecke form and let $\chi$ be a character of conductor $q^3$. Let $\mathcal B=\mathcal B_k(16q^4)$ be an orthogonal basis of $S_k(16q^4)$ containing the given form $f$. Let $N\leq q^{\frac{9}{2}+\varepsilon}$ and set
$$
S:=\frac{\Gamma(k-1)}{(4\pi)^{k-1}}\sum_{g\in \mathcal B}\frac{1}{\|g\|_{16q^4}^2}\left|\sum_{n\in\mathbb Z}\lambda_g(n^2)\chi(n)h\left(\frac{n}{N}\right)\right|^2.
$$ 
In this notation the statement of Theorem \ref{thm2} translates to 
\begin{align}
\label{bdd}
S\ll_{k,\varepsilon} Nq^{\varepsilon}.
\end{align}
Rest of the paper is devoted to proving this bound. \\

Opening the absolute square and interchanging the order of summation we arrive at
$$
S=\mathop{\sum\sum}_{n,m\in\mathbb Z}\chi(n)\overline{\chi(m)}h\left(\frac{n}{N}\right)h\left(\frac{m}{N}\right)\left[\frac{\Gamma(k-1)}{(4\pi)^{k-1}}\sum_{g\in \mathcal B}\frac{\lambda_g(n^2)\lambda_g(m^2)}{\|g\|_{16q^4}^2}\right].
$$ 
Now to the innermost sum we apply the Petersson formula \eqref{pform}. The contribution from the diagonal is dominated by
$\sum_nh\left(\frac{n}{N}\right)^2\ll N$, which is satisfactory for our purpose. Now we turn our attention to the off-diagonal which is given by
\begin{align}
\label{offdiag}
S_O=\sum_{\substack{c=1\\16|c}}^{\infty}\frac{1}{q^4c}\sum_n\sum_m\chi(n)\overline{\chi(m)}h\left(\frac{n}{N}\right)h\left(\frac{m}{N}\right)S(n^2,m^2;q^4c)J_{k-1}\left(\frac{4\pi nm}{q^4 c}\right).
\end{align}
Using a smooth partition of unity we break the sum over $c$ into dyadic blocks and analyse the contribution of each blocks
\begin{align}
\label{offdiag-dy}
S_O(C)=\sum_{\substack{c=1\\16|c}}^{\infty}\frac{1}{q^4c}\sum_n\sum_m\chi(n)\overline{\chi(m)}h\left(\frac{n}{N}\right)h\left(\frac{m}{N}\right)S(n^2,m^2;q^4c)J_{k-1}\left(\frac{4\pi nm}{q^4 c}\right)G\left(\frac{c}{C}\right).
\end{align}
Here $G(x)$ is a smooth function on $(0,\infty)$ with compact support. The transition range for the Bessel function is marked by $C\sim N^2/q^4$, and  we define $B$ by setting $C=\frac{N^2}{q^4B}$. Of course the real challenge lies in dealing with the case where $C$ is near the transition range. However for smaller values of $C$ there are complications arising from the oscillation of the Bessel function. \\

We write $c=q^rc'$ where $(c',q)=1$. Then the Kloosterman sum splits as
$$
S(n^2,m^2;q^4c)=S(\overline{c'}n^2,\overline{c'}m^2;q^{4+r})S(\overline{q^{4+r}}n^2,\overline{q^{4+r}}m^2;c').
$$
Observe that in \eqref{offdiag-dy} $nm$ is coprime with $q$ due to the presence of the character $\chi$. So it follows that
$$
S(\overline{c'}n^2,\overline{c'}m^2;q^{4+r})=S(\overline{c'}nm,\overline{c'}nm;q^{4+r})=2\left(\frac{c'nm}{q}\right)^rq^{2+\frac{r}{2}}\:\text{Re} \:\varepsilon_{q^{4+r}}e\left(\frac{2\overline{c'}nm}{q^{4+r}}\right)
$$
where $\varepsilon_{q^{4+r}}$ is the sign of the quadratic Gauss sum modulo $q^{4+r}$. With this $S_O(C)$ splits into a sum of two similar terms, each of which is an infinite sum parameterized by $r$. A representative term in this sum is given by
\begin{align}
\label{sum1}
\frac{1}{q^{2+\frac{r}{2}}}\sum_{\substack{(c,q)=1\\16|c}}&\frac{1}{c}G\left(\frac{q^rc}{C}\right)\sum_n\sum_m\chi(n)\overline{\chi(m)}\left(\frac{cnm}{q}\right)^re\left(\frac{2\overline{c}nm}{q^{4+r}}\right)S(\overline{q^{4+r}}n^2,\overline{q^{4+r}}m^2;c)\\
\nonumber &\times h\left(\frac{n}{N}\right)h\left(\frac{m}{N}\right)J_{k-1}\left(\frac{4\pi nm}{q^{4+r} c}\right).
\end{align}

\begin{remark}
For the larger values of $r$ we can estimate the sum trivially using the bounds of the Bessel function. Indeed using the bound \eqref{bessel-bd0} we get that the above sum is dominated by 
$$
\ll \frac{1}{q^{2}}\sum_{\substack{c\sim C\\q^r|c}}\frac{c^{\varepsilon}}{\sqrt{c}}\mathop{\sum\sum}_{n,m\ll N} (n^2,m^2,c)\min\left\{\left(\frac{N^2}{q^{4} c}\right)^{k-1},\left(\frac{N^2}{q^{4} c}\right)^{-1/2}\right\},
$$
which is smaller than $N^3/q^{4+r}$, and hence satisfactory for our purpose if $r\geq 5$. 
\end{remark}

We will establish that \eqref{sum1} is dominated by
\begin{align}
\label{bdd2}
\ll Nq^{\varepsilon}
\end{align}
which is what we need to claim \eqref{bdd}. To prove this bound we start by considering the sums over $n$ and $m$. The modulus of the sum is given by $q^{4+r}c$. But we can reduce the modulus to $q^3c$ by applying reciprocity.
Indeed we have
$$
e\left(\frac{2\overline{c}nm}{q^{4+r}}\right)=e\left(\frac{-2\overline{q^{4+r}}nm}{c}\right)
e\left(\frac{2nm}{q^{4+r}c}\right).
$$
We club the second factor with the Bessel function in \eqref{sum1}. Using the expression \eqref{bessel-split} we see that 
$$
e\left(\frac{2nm}{q^{4+r}c}\right)J_{k-1}\left(\frac{4\pi nm}{q^{4+r}c}\right)=e\left(\frac{4nm}{q^{4+r}c}\right)W_k\left(\frac{2nm}{q^{4+r}c}\right)+\bar W_k\left(\frac{2nm}{q^{4+r}c}\right).
$$
The second factor on the right hand side is without oscillation, and hence is more tamed compared to the first factor. We shall now continue our analysis with the first factor. The analysis with the second factor, which we are omitting, is much simpler. At the end it turns out that the bound that we obtain for the contribution of the second factor is better than that of the first factor. \\

We define 
$$
\Phi(x,y;c)=h\left(x\right)h\left(y\right)e\left(\frac{4xy}{c}\right)W_{k}\left(\frac{2xy}{c}\right)
$$
and set (for $r\leq 4$)
\begin{align}
\label{tr}
\mathcal T_{r,C}=\frac{1}{q^{2+\frac{r}{2}}}\sum_{\substack{c=1\\q\nmid c}}^{\infty}\frac{T_r(16c)}{16c}G\left(\frac{q^rc}{C}\right),
\end{align}
where $T_r(c)$ is defined by
\begin{align*}
\mathop{\sum\sum}_{(n,m)\in \mathbb Z^2}\chi(n)\overline{\chi(m)}\left(\frac{cnm}{q}\right)^re\left(\frac{-2\overline{q^{4+r}}nm}{c}\right)S(\overline{q^{4+r}}n^2,\overline{q^{4+r}}m^2,c)\Phi\left(\frac{n}{N},\frac{m}{N};\frac{q^{4+r}c}{N^2}\right).
\end{align*}
Now we are ready to apply the Poisson summation formula on both the sums over $n$ and $m$. Since the modulus of the sum after reciprocity is $q^3c$ which in the transition range is of the size $N^2/q^{1+r}$, the length of the dual sum in the transition range is given by $N/q^{1+r}$. So we will have a reduction in the length of the sum. Moreover there is a structural gain as the Kloosterman sum dissolves to yield character sums of manageable complexity (almost like Gauss sum). \\

\begin{lemma}
\label{poisson}
We have
\begin{align}
\label{trrt}
T_r(c)=\frac{N^2}{q^6c^2}\mathop{\sum\sum}_{(n,m)\in\mathbb Z^2}D_r(n,m;c)I_r(n,m;c),
\end{align}
where the character sum $D_r(n,m;c)$ is given by
$$
\mathop{\sum\sum}_{\alpha,\beta \bmod q^3c}\chi(\alpha)\overline{\chi(\beta)}\left(\frac{\alpha\beta}{q}\right)^re\left(\frac{-2\overline{q^{4+r}}\alpha\beta}{c}\right)S(\overline{q^{4+r}}\alpha^2,\overline{q^{4+r}}\beta^2,c)e\left(\frac{\alpha n+ \beta m}{q^3c}\right),
$$
and the integral $I_r(n,m;c)$ is given by
$$
\mathop{\iint}_{\mathbb R^2}h\left(x\right)h\left(y\right)e\left(\frac{4xyN^2}{q^{4+r}c}\right)W_k\left(\frac{2xyN^2}{q^{4+r}c}\right)e\left(\frac{-(nx+my)N}{q^3c}\right)dxdy.
$$
\end{lemma}
\begin{proof}
First we break the sum over $n$ and $m$ into congruence classes modulo $q^3c$ to get
\begin{align*}
\mathop{\sum\sum}_{\alpha,\beta \bmod q^3c}\chi(\alpha)\overline{\chi(\beta)}&\left(\frac{\alpha\beta}{q}\right)^re\left(\frac{-2\overline{q^{4+r}}\alpha\beta}{c}\right)S(\overline{q^{4+r}}\alpha^2,\overline{q^{4+r}}\beta^2,c)\\
&\times \mathop{\sum\sum}_{(n,m)\in\mathbb Z^2}\Phi\left(\frac{\alpha+nq^3c}{N},\frac{\beta+mq^3c}{N};\frac{q^{4+r}c}{N^2}\right).
\end{align*}
Now applying Poisson summation formula and making a change of variables, we get that the inner sum over $(n,m)$ is given by
$$
\frac{N^2}{q^6c^2}\mathop{\sum\sum}_{(n,m)\in\mathbb Z^2}e\left(\frac{\alpha n+ \beta m}{q^3c}\right)
\mathop{\iint}_{\mathbb R^2}\Phi\left(x,y;\frac{q^{4+r}c}{N^2}\right)
e\left(\frac{-(nx+my)N}{q^3c}\right)dxdy.
$$
The lemma follows after rearranging the order of summations and integration.
\end{proof}
\begin{remark}
We conclude this section with a simple observation. Using integration-by-parts on $I_r(n,m;c)$ and the analytic properties of the Bessel function we get that we can basically ignore the contribution coming from the terms in \eqref{tr} with $c$ large, e.g. $c>Q=q^{2010}$, or from the terms in \eqref{trrt} with $\min\{|n|,|m|\}>Q$. 
\end{remark}

%========================================================================================================================
%========================================================================================================================

\section{The character sum $D_r(n,m;c)$}
\label{sec-char}

In this section we will explicitly compute the character sum $D_r(n,m;c)$ which appears in Lemma~\ref{poisson}. First since $(c,q)=1$ we have the factorization
\begin{align}
\label{char-sum-d-r}
D_r(n,m;c)=A_r(\bar cn,\bar cm;q^3)B_r(\overline{q^3}n,\overline{q^3}m;c)
\end{align}
where
$$
A_r(n,m;q^3)=\mathop{\sum\sum}_{\alpha,\beta \bmod q^3}\chi(\alpha)\overline{\chi(\beta)}\left(\frac{\alpha\beta}{q}\right)^re\left(\frac{\alpha n+ \beta m}{q^3}\right),
$$
and
\begin{align}
\label{char-b}
B_r(n,m;c)=\mathop{\sum\sum}_{\alpha,\beta \bmod c}e\left(\frac{-2\overline{q^{4+r}}\alpha\beta}{c}\right)S(\overline{q^{4+r}}\alpha^2,\overline{q^{4+r}}\beta^2,c)e\left(\frac{\alpha n+\beta m}{c}\right).
\end{align}
The first character sum $A_r(\bar cn,\bar cm;q^3)$, in fact, does not depend on $c$. This is because of the nice feature that both $\chi$ and $\overline{\chi}$ appear in the sum. Let 
\begin{align}
\label{an-def}
a_n=\sum_{\alpha\bmod q^3}\chi(\alpha)\left(\frac{\alpha}{q}\right)^re\left(\frac{\alpha n}{q^3}\right).
\end{align}
Since $\chi$ is a primitive character of modulus $q^3$, it follows that $a_n=0$ whenever $q|n$. We conclude that 
$$
a_n=\overline{\chi(n)}\left(\frac{n}{q}\right)^ra_1.
$$ 
Moreover we note that we have $|a_1|\leq q^{3/2}$. The following lemma is obvious.
\begin{lemma}
\label{char1}
We have
$$
A_r(\bar cn,\bar cm;q^3)=\overline{\chi(n)}\chi(-m)\left(\frac{-nm}{q}\right)^r|a_1|^2.
$$\\
\end{lemma}

The other character sum is little delicate, but nevertheless it can also be evaluated. To this end, for any odd integer $d$, we define
$$
C(n,m;d)=\sideset{}{^\star}\sum_{\substack{a\bmod d\\an\equiv m \bmod d}}\left(\frac{a}{d}\right).
$$
Clearly $C(n,m;d)$ is multiplicative as a function of $d$. The following lemma can be proved quite easily.
\begin{lemma}
\label{char11}
Let $p$ be a prime. Then we have
$$
C(n,m;p^{\ell})=
\begin{cases}
\left(\frac{n^*m^*}{p^{\ell}}\right)p^j &\text{if $n=p^jn^*$, $m=p^jm^*$, $p\nmid n^*m^*$, $j<\ell$},\\
\phi(p^{\ell}) &\text{if $\ell$ is even, $p^{\ell}|(n,m)$},\\
0 &\text{otherwise}.
\end{cases}
$$\\
\end{lemma}

To study $B_r(n,m;c)$ we also need the analogue of $C(n,m;d)$ for $d$ a power of $2$. This is slightly more involved as there are three non-trivial quadratic characters, namely $\chi_{-4}$, $\chi_8$ and $\chi_{-8}$, which have a 2-primary modulus. For $\eta\geq 4$, let 
$$
C_{\pm}(n,m; 2^{\eta})=\sideset{}{^\star}\sum_{\substack{a\bmod 2^{\eta}\\an\equiv m \bmod 2^{\eta}}}\psi_{\pm}(a).
$$
Here 
$$
\psi_{+}=\begin{cases}\chi_0 &\text{if $\eta$ is even,}\\\chi_{8} &\text{if $\eta$ is odd,}\end{cases}
$$ 
and $\psi_{-}=\psi_{+}\chi_{-4}$.
\begin{lemma}
\label{char12}
For $\eta$ even, we have
$$
C_{\pm}(n,m; 2^{\eta})=
\begin{cases}
\psi_{\pm}(n^*m^*)2^j &\text{if $n=2^jn^*$, $m=2^jm^*$, $2\nmid n^*m^*$, $j+1<\eta$},\\
\phi(2^{\eta}) &\text{if the character is trivial and $2^{\eta-1}|(n,m)$},\\
0 &\text{otherwise}.
\end{cases}
$$
For $\eta$ odd, we have
$$
C_{\pm}(n,m; 2^{\eta})=
\begin{cases}
\psi_{\pm}(n^*m^*)2^j &\text{if $n=2^jn^*$, $m=2^jm^*$, $2\nmid n^*m^*$, $j+2<\eta$},\\
0 &\text{otherwise}.
\end{cases}
$$\\
\end{lemma}

Now for general modulus $c=d2^{\eta}$, with $d$ odd, we define 
$$
C_{\pm}(n,m;c)=\sideset{}{^\star}\sum_{\substack{a\bmod c\\an\equiv m \bmod c}}\left(\frac{a}{d}\right)\psi_{\pm}(a).
$$
Clearly we have the factorization $C_{\pm}(n,m;c)=C(n,m;d)C_{\pm}(n,m;2^{\eta})$. From the above lemmas we conclude that the character sum $C(n,m;d)$ splits in the following fashion: Let $(n,m)=\delta$, and write $n=\delta n^*$ and $m=\delta m^*$. Also write $c=c_1c_2$ with $c_1|(nm)^{\infty}$, $(c_2,nm)=1$ and $c_1=c_{11}c_{12}^2$ with $c_{11}$ square-free. It follows that the sum vanishes unless $c_1|\delta^{\infty}$. In this case we have the following:
\begin{corollary}
\label{cor}
Using the above notations we have
$$
C_{\pm}(n,m;c)=\left(\frac{n^*m^*}{c_{11}c_2}\right)\psi_{\pm}(n^*m^*)\tilde C(\delta;c_1),
$$
where $\tilde C(\delta;c_1)$ depends only on $\delta$, the gcd of $n$ and $m$, and on $c_1$. Moreover we have 
$$
\left|\tilde C(\delta;c_1)\right|\leq (\delta_2^2,c_1),
$$
where $\delta=\delta_1\delta_2^2$ with $\delta_1$ square-free.
\end{corollary}

Now we are ready to evaluate the character sum $B_r(n,m;c)$. Set $\varepsilon_d = 1$ if $d\equiv 1 \bmod 4$, and $\varepsilon_d = i$ if $d\equiv 3 \bmod 4$. 

\begin{lemma}
\label{char2}
Let $c=2^{\eta}d$ with $d$ odd, $(q,d)=1$ and $\eta\geq 4$. Then $B_r(n,m;c)=0$ unless $4|(n,m)$, and $(m,c)=(n,c)$. Write $n=2n'$ and $m=2m'$. Then we have 
$$
B_r(n,m;c)=c^{\frac{3}{2}}\varepsilon_de\left(\frac{q^{4+r}n'm'}{c}\right)\left\{C_{+}(q^{4+r}n,-m;c)+i\chi_{-4}(d)C_{-}(q^{4+r}n,-m;c)\right\}.
$$
\end{lemma}
\begin{proof}
Applying a change of variables in \eqref{char-b}, we get
$$
B_r(n,m;c)=\mathop{\sum\sum}_{\alpha,\beta \bmod c}e\left(\frac{-2\alpha\beta}{c}\right)S(\alpha^2,\beta^2;c)e\left(\frac{\alpha q^{4+r}n+\beta m}{c}\right).
$$
Opening the Kloosterman sum  and rearranging the order of summations, we get 
$$
\sideset{}{^\star}\sum_{a\bmod c}\sum_{\beta\bmod c}e\left(\frac{\overline{a}\beta^2+\beta m}{c}\right)
\sum_{\alpha\bmod c}e\left(\frac{a\alpha^2+(q^{4+r}n-2\beta)\alpha}{c}\right).
$$
We have $c=d2^{\eta}$ with $d$ odd, and $\eta\geq 4$. The inner sum over $\alpha$ vanishes unless $n$ is even, say $n=2n'$, in which case it is given by (see \cite{Bl})
$$
\varepsilon_{d} \sqrt{c} \left(\frac{a}{d}\right)e\left(\frac{-\overline{a}(q^{4+r}n'-\beta)^2}{c}\right)\times
\begin{cases}
(1+i\chi_{-4}(ad)), &\text{if $\eta$ is even},\\
(\chi_8(a)+i\chi_{-4}(d)\chi_{-8}(a)), &\text{if $\eta$ is odd}.
\end{cases}
$$
The sum over $\beta$ now gives
$$
\sum_{\beta\bmod c}e\left(\frac{\overline{a}\beta^2+\beta m-\overline{a}(q^{4+r}n'-\beta)^2}{c}\right)=
e\left(\frac{-\overline{a}(q^{4+r}n')^2}{c}\right)
\sum_{\beta\bmod c}e\left(\frac{(m+\overline{a}q^{4+r}n)\beta}{c}\right).
$$
The last sum vanishes unless 
$$
m+\overline{a}q^{4+r}n\equiv 0 \bmod c,
$$
in which case the sum is given by $c$. The congruence condition forces the equality $(m,c)=(n,c)$. In particular $m=2m'$ is even. It follows that $B_r(n,m;c)$ is given by
$$
c^{3/2}\varepsilon_{d} \sideset{}{^\star}\sum_{\substack{a\bmod c\\m+\bar aq^{4+r}n\equiv 0\bmod c}}\left(\frac{a}{d}\right)e\left(\frac{-\bar a(q^{4+r}n')^2}{c}\right)\times\begin{cases}
\left(1+i\chi_{-4}(ad)\right) &\text{even $\eta$},\\
\left(\chi_8(a)+i\chi_{-4}(d)\chi_{-8}(a)\right) &\text{odd $\eta$}.
\end{cases}
$$
This leads us to consider the sum 
$$
\sideset{}{^\star}\sum_{\substack{a\bmod c\\m+\bar aq^{4+r}n\equiv 0\bmod c}}\left(\frac{a}{d}\right)\psi_{\pm}(a)e\left(\frac{-\bar a(q^{4+r}n')^2}{c}\right),
$$
which we want to express in terms of $C(n,m;c)$. The congruence condition does not uniquely determine $a$, but the product $\bar aq^{4+r}n'$ gets determined (in terms of $m'$) modulo $c'=2^{\eta-1}d$. The above character sum reduces to
$$
e\left(\frac{q^{4+r}n'm'}{c}\right)\left\{\;\sideset{}{^\star}\sum_{\substack{a\bmod c\\aq^{4+r}n'\equiv -m' \bmod c}}\left(\frac{a}{d}\right)\psi_{\pm}(a)+\varepsilon(n')
\sideset{}{^\star}\sum_{\substack{a\bmod c\\aq^{4+r}n'\equiv -m'+c' \bmod c}}\left(\frac{a}{d}\right)\psi_{\pm}(a)\right\},
$$
where $\varepsilon(n')=-1$ if $n'$ is odd and $\varepsilon(n')=1$ if $n'$ is even. Consequently, when $n'$ is even, i.e. $4|(n,m)$, the above sum reduces to 
$$
e\left(\frac{q^{4+r}n'm'}{c}\right)C_{\pm}(q^{4+r}n,-m;c).
$$
On the other hand if $n'$ is odd, i.e. $2\|n$ and $2\|m$, the sum vanishes for all $\eta\geq 4$. The lemma follows.
\end{proof}

\begin{corollary}
\label{cor-b-sum}
Let $c=2^{\eta}d$ with $d$ odd, $(q,d)=1$ and $\eta\geq 4$. Then $B_r(n,m;c)=0$ unless $4|(n,m)$, and $(m,c)=(n,c)$. Write $n=2n'$ and $m=2m'$. Then we have 
$$
B(\overline{q^3}n,\overline{q^3}m;c)=c^{\frac{3}{2}}\varepsilon_de\left(\frac{\overline{q^2}q^{r}n'm'}{c}\right)\left\{C_{+}(q^{4+r}n,-m;c)+i\chi_{-4}(d)C_{-}(q^{4+r}n,-m;c)\right\}.
$$\\
\end{corollary}

%====================================================================================================
%====================================================================================================

\section{The integral $I_r\left(n,m;c\right)$}
\label{integral-00} 

Now we will study the integral 
\begin{align}
\label{int-again}
I_r\left(n,m;c\right)=\mathop{\iint}_{\mathbb R^2}h\left(x\right)h\left(y\right)e\left(\frac{4xyN^2}{q^{4+r}c}\right)W_k\left(\frac{2xyN^2}{q^{4+r}c}\right)e\left(\frac{-(nx+my)N}{q^3c}\right)dxdy.
\end{align}
which appears in Lemma \ref{poisson}. We will now obtain bounds for this integral using repeated integration by parts. Recall that $cq^r\sim C$ and $C=\frac{N^2}{q^4B}$. So we will write $c=Cz/q^r$ for some $z\in [1,2]$. Differentiating the first four factors and integrating the last factor we get
$$
I_r\left(n,m;\frac{Cz}{q^r}\right)\ll_{j_1,j_2} \left(1+\frac{N^2}{q^4C}\right)^{j_1+j_2}\left(\frac{q^3C}{|n|Nq^r}\right)^{j_1}\left(\frac{q^3C}{|m|Nq^r}\right)^{j_2} \;\;\;\;\text{for}\;\;j_1, j_2\geq 0.
$$
So it follows that the above integral is negligibly small (i.e. $O(q^{-L})$ for any $L>0$) unless 
\begin{align}
\label{nm-bound}
|n|, |m| \ll \frac{N}{q^{1+r}}\left(B^{-1}+1\right)q^{\varepsilon}.
\end{align}\\

For $B>q^{\varepsilon}$ there is oscillation in the third factor on the right hand side of \eqref{int-again}. In this case we may also do repeated integration by parts by integrating the third factor and differentiating the other factors. This process yields the bound
$$
I_{r}\left(n,m;\frac{Cz}{q^r}\right)\ll_{j_1,j_2} \left(1+\frac{|n|Nq^r}{q^3C}\right)^{j_1}\left(1+\frac{|m|Nq^r}{q^3C}\right)^{j_2}\left(\frac{q^4C}{N^2}\right)^{j_1+j_2}\;\;\;\;\text{for}\;\;j_1, j_2\geq 0.
$$
Since we are assuming that $C=\frac{N^2}{q^4B}< \frac{N^2}{q^{4+\varepsilon}}$, it follows that the integral is arbitrarily small unless $|n|, |m| \gg \frac{N}{q^{1+r}}q^{-\varepsilon}$. 
\begin{lemma}
\label{int-000}
Suppose $C<\frac{N^2}{q^{4+\varepsilon}}$ (or in other words $B>q^{\varepsilon}$) then the integral $I_{r}\left(n,m;\tfrac{Cz}{q^r}\right)$ is negligibly small unless 
$$
|n|, |m| \in \left[\frac{N}{q^{1+r}}q^{-\varepsilon},\frac{N}{q^{1+r}}q^{\varepsilon}\right].
$$\\
\end{lemma}

We will now show that the bound that we have obtained so far is satisfactory for $r\geq 2$ if the weight $k>2$. Indeed using the bound from \eqref{bessel-bd0} we get
$$
I_r\left(n,m;\frac{Cz}{q^r}\right)\ll \min\{B^{-\frac{1}{2}}, B^{k-1}\}.
$$ 
Also, from our computations in the previous section it follows that
$$
D_r(n,m;c)\ll (n,m,c)q^3c^{\frac{3}{2}}.
$$
Hence (as we are momentarily assuming $k>2$)
\begin{align*}
\mathcal T_{r,C}\ll &\frac{N^2}{q^{8+\frac{r}{2}}}\sum_{c\sim C/q^r}\frac{1}{c^3}\mathop{\sum\sum}_{1\leq |n|,|m| < Nq^{\varepsilon}(1+B^{-1})/q^{1+r}}|D_r(n,m;c)I_r(n,m;c)|+q^{-2010}\\
\ll &\frac{N^4q^{\varepsilon}}{q^{7+\frac{5r}{2}}}\left(1+\frac{1}{B^2}\right)\min\{B^{-\frac{1}{2}}, B^2\}\sum_{c\sim C/q^r}\frac{1}{c^{\frac{3}{2}}}+q^{-2010}\ll \frac{N^3q^{\varepsilon}}{q^{5+\frac{5r}{2}}}.
\end{align*}
\begin{lemma}
\label{med-r}
For $r\geq 2$, we have
$$
\mathcal T_{r,C}\ll Nq^{\varepsilon}.
$$
\end{lemma}

\begin{remark}
We will not use this lemma. The ultimate bound that we obtain in the following sections, works equally well for all values of $r$. Also the restriction on the weight $k>2$ is not necessary for our final bound.
\end{remark}
%========================================================================================================================
%========================================================================================================================

\section{Reciprocity and Poisson summation - II}
\label{sec-poi}

In Section \ref{sec-char}, we explicitly computed the character sum which appears in Lemma \ref{poisson}. Substituting this explicit form of the character sum in \eqref{tr}, and ignoring the negligible contribution that comes from the large values of $c$ and the large values of the gcd $(n,m)$ (as we have noted after Lemma \ref{poisson}), we are basically left with the job of analysing sums of the type 
\begin{align}
\label{tr'}
\mathcal T_{r,C}^\star=\frac{N^2}{q^{5+\frac{r}{2}}}\sum_{\substack{\delta\leq Q\\q\nmid \delta}}
\sum_{\substack{c_1\leq Q\\c_1|(2\delta)^{\infty}}}(\delta_2^2,c_1)\mathop{\sum\sum}_{\substack{1\leq u,v\ll Q\\u,v|(2\delta)^{\infty}\\(u,v)=1}}\Bigl|T_{r,C}^\star(\delta,c_1,u,v)\Bigr|
\end{align}
where $\delta=\delta_1\delta_2^2$ with $\delta_1$ square-free,
\begin{align}
\label{tr''}
T_{r,C}^\star(\delta,c_1,u,v)=\mathop{\sum\sum}_{\substack{n,m=1\\(n,m)=1\\(nm,\delta q)=1\\ n,m\equiv 1 \bmod 4}}^{\infty}
\overline{\chi(n)}\chi(m)\left(\frac{nm}{q^r}\right)S_{r,c_1}^\star(\delta un,\delta vm),
\end{align}
and 
\begin{align}
\label{S_r}
S_{r,c_1}^\star(n,m)=\sum_{\substack{c_2=1\\(c_2,2qnm)=1}}^{\infty}\left(\frac{nmq^r}{c_2}\right)e\left(\frac{\overline{q^2}q^{r}nm}{c_1c_2}\right)\frac{I_{r}(2n,2m;c_1c_2)}{(c_1c_2)^{3/2}}G\left(\frac{c_1c_2q^r}{C}\right).
\end{align}
(Here $Q=q^{2010}$.) The condition that $q\nmid \delta$ in \eqref{tr'} is justified by the fact that $a_n=0$ (see \eqref{an-def}) whenever $q|n$. (This also implies that there is no zero frequency $nm=0$ contribution to worry about.) The assumption on $n$ and $m$, that they are $\equiv 1 \bmod 4$, is made to simplify some of the standard complications related to the prime $2$. In general we may take out the $2$-primary part from $n$ or $m$, and then split the sum into four parts depending on the possible congruence classes modulo $4$. Then for each sum we follow the same steps that we take below. Also note that for notational simplicity we are just focusing on the contribution from the nonnegative $n$ and $m$. \\

Our next step is an application of the Poisson summation formula on the sum over $c_2$ in \eqref{S_r}. To this end we have to first apply reciprocity
$$
e\left(\frac{\overline{q^2}q^{r}nm}{c_1c_2}\right)=
e\left(\frac{-\overline{c_1c_2}q^{r}nm}{q^2}\right)
e\left(\frac{q^{r}nm}{q^2c_1c_2}\right).
$$
We will include the last factor in our smooth function. Set
$$
\mathcal I_{r}(n,m;c)=e\left(\frac{q^{r}nm}{q^2c}\right)\frac{I_{r}(2n,2m;c)}{c^{3/2}}G\left(\frac{cq^r}{C}\right).
$$ 
Now consider the sum 
\begin{align}
\label{S_r*}
S_{r,c_1}^\star(n,m)=\sum_{\substack{c_2\in\mathbb Z\\(c_2,2qnm)=1}}\left(\frac{c_2}{nmq^r}\right)e\left(\frac{-\overline{c_1c_2}q^{r}nm}{q^2}\right)\mathcal I_{r}(n,m;c_1c_2).
\end{align}
Here we have applied quadratic reciprocity. This is one of the places where we use the assumption that $n,m\equiv 1 \bmod 4$. We will use this yet another time in the evaluation of the character sum that appears in our next result.
\begin{lemma}
\label{poisson2}
We have
$$
S_{r,c_1}^\star(n,m)=\frac{q^{\frac{r}{2}}}{2c_1q^{2}\sqrt{C}[n,m]}\sum_{c_2\in\mathbb Z}E_{r,c_1}(c_2;n,m)\tilde{\mathcal I}_{r,c_1}(c_2;n,m)
$$
where the character sum is given by 
$$
E_{r,c_1}(c_2;n,m)=\sideset{}{^\star}\sum_{\beta \bmod 2q^2[n,m]}\left(\frac{\beta }{nmq^r}\right)e\left(\frac{-\overline{c_1\beta}q^{r}nm}{q^2}\right)e\left(\frac{\beta c_2}{2q^2[n,m]}\right),
$$
and
\begin{align*}
\tilde{\mathcal I}_{r,c_1}(c_2;n,m)=\int G(z)I_{r}\left(2n,2m;\frac{Cz}{q^r}\right)
e\left(\frac{q^{2r}nm}{q^2Cz}-\frac{c_2Cz}{2q^{2+r}c_1[n,m]}\right)\frac{dz}{z^{\frac{3}{2}}}.
\end{align*}
\end{lemma}
\begin{proof}
Breaking the sum in \eqref{S_r*} into congruence classes modulo $2q^2[n,m]$ (here $[.,.]$ denotes lcm) we get
$$
\sideset{}{^\star}\sum_{\beta \bmod 2q^2[n,m]}\left(\frac{\beta }{nmq^r}\right)e\left(\frac{-\overline{c_1\beta}q^{r}nm}{q^2}\right)\sum_{c_2\in\mathbb Z}\mathcal I_{r,j}(n,m;c_1(\beta+c_22q^2[n,m])).
$$
We apply the Poisson summation formula to the inner sum to replace it with 
$$
\sum_{c_2\in\mathbb Z}\int\mathcal I_{r,j}(n,m;c_1(\beta+2zq^2[n,m]))e(-zc_2)dz.
$$
After a change of variable the integral reduces to
\begin{align*}
\frac{q^{\frac{r}{2}}}{2c_1q^{2}\sqrt{C}[n,m]}e\left(\frac{\beta c_2}{2q^2[n,m]}\right)\tilde{\mathcal I}_{r,j,c_1}(c_2;n,m).
\end{align*}
The lemma now follows by rearranging the sums.
\end{proof}

%========================================================================================================================
%========================================================================================================================

\section{The character sum $E_{r,c_1}(c_2;n,m)$}
\label{sec-char2}

Next we will explicitly evaluate the character sum which appears in Lemma \ref{poisson2}. For \eqref{tr'}, we only require to consider the character sum $E_{r,c_1}(c_2;u\delta n,v\delta m)$ where $uv|(2\delta)^{\infty}$, $(u,v)=(n,m)=1$, $n,m\equiv 1\bmod{4}$ and $(nm,\delta q)=1$. In this case the sum splits into a product given by 
\begin{align}
\label{split-sum}
\left(\frac{\delta}{nm}\right)\left[\:\sideset{}{^\star}\sum_{\beta \bmod 2\delta uv}\left(\frac{\beta }{uv}\right)e\left(\frac{\beta c_2}{2\delta uv}\right)\right]&\left[\:\sideset{}{^\star}\sum_{\beta \bmod nm}\left(\frac{\beta}{nm}\right)e\left(\frac{\beta c_2}{nm}\right)\right]\\
\nonumber &\times\left[\:\sideset{}{^\star}\sum_{\beta \bmod q^2}\left(\frac{\beta }{q^r}\right)e\left(\frac{-\overline{c_1\beta}q^{r}\delta\ell+\overline{2\ell}\beta c_2}{q^2}\right)\right],
\end{align}
where $\ell=[\delta un,\delta vm]=\delta uvnm$. The first two sums are just Gauss sums and the last sum is a generalized Kloosterman sum. We denote the first sum by $g_{\delta,u,v}^\star(c_2)$, and note the following bound:
\begin{lemma}
\label{bound-for-g}
Let $\delta=\delta_1\delta_2^2$ and $c_2=c_{21}c_{22}^2$ with $\delta_1$, $c_{21}$ square-free. Then we have 
$$
g_{\delta,u,v}^\star(c_2)\ll uv\delta_2(\delta_1,c_{21})(\delta_1\delta_2,c_{22}).
$$
\end{lemma}
\begin{proof}
First using multiplicativity we reduce to the case where $\delta$, $u$ and $v$ are powers of a given prime $p$. Then we use well-known bounds for the Gauss sums and Ramanujan sums. Finally we need to verify that the power of the prime $p$ which appears on the right hand side is sufficiently large. For this we need to consider several cases. We prefer to omit the details. 
\end{proof}

The middle sum in \eqref{split-sum} is given by $g(n,c_2)g(m,c_2)$ where $g(n,c)$ stands for the usual Gauss sum 
$$
g(n,c)=\sideset{}{^\star}\sum_{\beta \bmod n}\left(\frac{\beta }{n}\right)e\left(\frac{\beta c}{n}\right).
$$
The other character sum modulo $q^2$, after a change of variable is given by 
$$
\left(\frac{\delta uvnm}{q^r}\right)\sideset{}{^\star}\sum_{\beta \bmod q^2}\left(\frac{\beta }{q^r}\right)e\left(\frac{-\overline{c_1\beta}q^{r}\delta+\overline{2}\beta c_2}{q^2}\right)=\left(\frac{\delta uvnm}{q^r}\right)S_r(\bar 2c_2,-\overline{c_1}q^r\delta;q^2).
$$
The Kloosterman type sum $S_r(\bar 2c_2,-\overline{c_1}q^r\delta;q^2)$ is free of $n$, $m$, and it is bounded above by $4(c_2,q^r,q^2)q$, which is the Weil bound. 
\begin{lemma}
\label{char31}
Let $u,v,n,m$ be as in \eqref{tr'}. Then
$$
E_{r,c_1}(c_2;u\delta n,v\delta m)=g_{\delta,u,v}^\star(c_2)g(n,c_2)g(m,c_2)\left(\frac{\delta uvnm}{q^r}\right)\left(\frac{\delta}{nm}\right)S_r(\bar 2c_2,-\overline{c_1}q^r\delta;q^2).
$$
\end{lemma}
 
\vspace{.5cm}

Substituting the explicit value of the character sum in Lemma \ref{poisson2}, and interchanging the order of summations, we get 
\begin{align}
\label{trr}
T_{r,C}^\star(\delta,c_1,u,v)=&\frac{1}{2q^{2-\frac{r}{2}}\sqrt{C}}\frac{1}{c_1\delta uv}\left(\frac{\delta uv}{q^r}\right)\sum_{c_2\in\mathbb Z}g_{\delta,u,v}^\star(c_2)S_r(\bar 2c_2,-\overline{c_1}q^r\delta;q^2)
\\
\nonumber &\times \mathop{\sum\sum}_{\substack{n,m=1\\(n,m)=1\\(nm,\delta q)=1\\ n,m\equiv 1 \bmod 4}}^{\infty}b_{\delta,\bar\chi}(n,c_2)b_{\delta,\chi}(m,c_2)\frac{\tilde{\mathcal I}_{r,c_1}(c_2;\delta un,\delta vm)}{nm}
\end{align}
where the new coefficients are given by
$$
b_{\delta,\psi}(n,c_2)=\psi(n)\left(\frac{\delta}{n}\right)g(n,c_2).
$$

%====================================================================================================
%====================================================================================================

\section{The integral $\tilde{\mathcal I}_{r,c_1}(c_2;n,m)$}
\label{integral}

In this section we will analyse the integral 
\begin{align}
\label{last-int-again}
\tilde{\mathcal I}_{r,c_1}(c_2;n,m)=\int G(z)I_{r}\left(2n,2m;\frac{Cz}{q^r}\right)
e\left(\frac{q^{2r}nm}{q^2Cz}-\frac{c_2Cz}{2q^{2+r}c_1[n,m]}\right)\frac{dz}{z^{\frac{3}{2}}}
\end{align} 
which appears in \eqref{trr} and is defined in Lemma~\ref{poisson2}. We will obtain bounds for this integral, which in particular will also give us the effective range for the $c_2$ sum in \eqref{trr}. Also we need to separate the variables $n$ and $m$ to pave the way for an application of the large sieve. Recall that we are taking $n,m>0$. \\

Let us temporarily write $\delta n$ and $\delta m$ in place of $n$ and $m$ respectively, and put the restriction $(n,m)=1$. We replace the integral representation \eqref{int-again} to obtain 
\begin{align*}
\tilde{\mathcal I}_{r,c_1}(c_2;\delta n,\delta m)=&\iiint h\left(x\right)h\left(y\right)G(z) W_k\left(\frac{2xyN^2}{q^{4}Cz}\right)\\
&\times e\left(\frac{4xyN^2+q^{2+2r}\delta^2nm-2\delta(nx+my)Nq^{1+r}}{q^4Cz}-\frac{c_2Cz}{2q^{2+r}c_1\delta nm}\right)\frac{dz}{z^{\frac{3}{2}}}dxdy.
\end{align*}
Using repeated integration by parts in the $z$-integral, and using the bounds for $n$ and $m$ from Section~\ref{integral-00}, it follows that the integral is negligibly small unless
$$
|c_2|\ll \frac{c_1q^{4+\varepsilon}}{\delta q^r}\left(\frac{1+B^3}{B}\right).
$$
This gives the effective range for the $c_2$ sum in \eqref{trr}. This is good enough for our purpose for $B<q^{\varepsilon}$. However for larger $B$ we will obtain a better range below.\\

To get a partial separation of the variables $n$ and $m$, we define new variables $x'=x/m$, $y'=y/n$ and $z'=z/nm$. With this change of variables $\tilde{\mathcal I}_{r,c_1}(c_2;\delta n,\delta m)$ reduces to
\begin{align*}
\sqrt{nm}\iiint h\left(mx'\right)h\left(ny'\right)G(nmz')W_k\left(\frac{2x'y'N^2}{q^{4}Cz'}\right)e\left(\frac{4B\Delta(x',y')}{z'}-\frac{c_2Cz'}{2q^{2+r}c_1\delta}\right)\frac{dz'}{z'^{\frac{3}{2}}}dx'dy',
\end{align*}
where $\Delta(x',y')=\left(x'-\frac{\delta q^{1+r}}{2N}\right)\left(y'-\frac{\delta q^{1+r}}{2N}\right)$. Now suppose $B>q^{\varepsilon}$, so that by Lemma \ref{int-000} we have $|m|\asymp N/q^{1+r}$ (upto a factor of size $q^{\varepsilon}$). Then for a given $x'$, by repeated integration by parts in the $y'$ integral we obtain that $\tilde{\mathcal I}_{r,c_1}(c_2;\delta n,\delta m)$ is negligibly small unless $\left|x'-\frac{\delta q^{1+r}}{2N}\right|\ll \frac{\delta q^{1+r+\varepsilon}}{NB}$. So we get that upto a negligibly error the integral $\tilde{\mathcal I}_{r,c_1}(c_2;\delta n,\delta m)$ is given by
\begin{align*}
\sqrt{nm}\mathop{\iint}_{\left|x'-\frac{\delta q^{1+r}}{2N}\right|\ll \frac{\delta q^{1+r}}{NB}q^{\varepsilon}} h\left(mx'\right)h\left(ny'\right)\int G(nmz')W_k\left(\tfrac{2x'y'B}{z'}\right)e\left(\tfrac{4B\Delta(x',y')}{z'}-\tfrac{c_2Cz'}{2q^{2+r}c_1\delta}\right)\frac{dz'}{z'^{\frac{3}{2}}}dx'dy'.
\end{align*} 
Now we can get a refined effective range for $c_2$ by integrating by parts the inner integral. Notice that the restriction on $x'$ implies that $B\Delta(x',y')/z' \ll q^{\varepsilon}$. So we get that the inner integral is negligibly small unless 
$$
|c_2|\ll \frac{c_1q^{4+\varepsilon}}{\delta q^r}B.
$$
Observe that compared to the previous bound we have saved an extra $B$.   \\

We summarize our findings in the following lemma. Here $\tilde G(s)$ and $\tilde h(s)$ denote the Mellin transform of the smooth compactly supported functions $G$ and $h$ respectively. 
\begin{lemma}
\label{all-about-int}
The integral $\tilde{\mathcal I}_{r,c_1}(c_2;\delta n,\delta m)$ is negligibly small unless 
$$
|n|, |m| \ll \frac{Nq^{\varepsilon}}{\delta q^{1+r}}\left(1+B^{-1}\right),\;\;\;\text{and}\;\;\;|c_2|\ll \frac{c_1q^{4+\varepsilon}}{\delta q^r}(B+B^{-1}).
$$
Moreover upto a negligible error term we have
\begin{align*}
\tilde{\mathcal I}_{r,c_1}(c_2;\delta n,\delta m)= \frac{1}{(2\pi i)^3}\mathop{\iiint}_{(\sigma_1),(\sigma_2),(\sigma_3)}\tilde h(s_1)\tilde h(s_2)\tilde G(s_3)\mathcal F_{r,c_1}(c_2;\delta,B;s_1,s_2,s_3)\frac{ds_1ds_2ds_3}{m^{s_1+s_3-\frac{1}{2}}n^{s_2+s_3-\frac{1}{2}}},
\end{align*} 
where
\begin{align*}
\mathcal F_{r,c_1}(c_2;\delta,B;s_1,s_2,s_3)=\mathop{\iiint}_{\mathcal R(B)} x^{-s_1}y^{-s_2}z^{-s_3}W_k\left(\tfrac{2xyB}{z}\right)e\left(\tfrac{4B\Delta(x,y)}{z}-\tfrac{c_2Cz}{2q^{2+r}c_1\delta}\right)\frac{dxdydz}{z^{\frac{3}{2}}}.
\end{align*} 
For $B<q^{\varepsilon}$ we take the region $\mathcal R(B)=[Q^{-1},Q]^3\cap \{2^{-1}\leq xy/z \leq 4\}$, and for $B\geq q^{\varepsilon}$ the region is obtained by putting the further restriction that $\left|x-\frac{\delta q^{1+r}}{2N}\right|\ll \frac{\delta q^{1+r}}{NB}q^{\varepsilon}$.\\
\end{lemma}

In the integral we will take the location of the contours to be $\sigma_1=\sigma_2=1$ and $\sigma_3=-\frac{1}{2}+\varepsilon$. For this choice we have
\begin{align}
\label{bound-for-F}
\mathcal F_{r,c_1}(c_2;\delta,B;s_1,s_2,s_3)\ll \min\{B^{-\frac{1}{2}},B^{k-1}\}\min\{1,B^{-1}\}q^{\varepsilon}\ll \min\{B^{-\frac{3}{2}},B^{k-1}\}q^{\varepsilon}.
\end{align}
This is obtained by trivially estimating the integrals over $x$, $y$ and $z$, taking into account the size of $W_k$ (see \eqref{bessel-bd0}) and the localization of $x$ for $B>q^{\varepsilon}$. Observe that the integral over $s_1$, $s_2$ and $s_3$ converges absolutely due to the rapid decay of the Mellin transforms as $|t|\rightarrow\infty$.

%===============================================================================================================
\section{The zero frequency}
\label{zero}

In this section we will show that the contribution coming from the zero frequency, i.e. $c_2=0$ in \eqref{trr}, is satisfactory for our purpose. 
\begin{lemma}
We have 
\begin{align*}
\mathop{\sum\sum}_{\substack{n,m=1\\(n,m)=1\\(nm,\delta q)=1\\ n,m\equiv 1 \bmod 4}}^{\infty}b_{\bar \chi}(n,0)b_{\chi}(m,0)\frac{\tilde{\mathcal I}_{r,c_1}(0;\delta un,\delta vm)}{nm}\ll \min\{B^{-\frac{3}{2}},B^{k-2}\}\frac{N}{q^{1+r}\delta\sqrt{uv}}q^{\varepsilon}.
\end{align*}
\end{lemma}
\begin{proof}
First observe that $g(n,0)$ vanishes unless $n$ is a square, in which case we have $g(n,0)=\phi(n)$. Thus it follows that the left hand side of the expression in the statement of the lemma, is bounded by
$$
\mathop{\sum\sum}_{n,m=1}^{\infty}|\tilde{\mathcal I}_{r,c_1}(0;\delta un^2,\delta vm^2)|.
$$
We can now conclude the lemma by appealing to Lemma \ref{all-about-int}, choosing $\sigma_1=\sigma_2=1$, $\sigma_3=-\frac{1}{2}+\varepsilon$, and applying the bound \eqref{bound-for-F}. 
\end{proof}
We note that
$$
S_r(0,-\overline{c_1}q^r\delta;q^2)=\begin{cases} \varepsilon_qq^{3/2} &\text{if $r=1$},\\
q(q-1) &\text{if $r>0$ and even},\\
0 &\text{otherwise}.
\end{cases}
$$
Thus it follows that the contribution of the zero frequency to $T_{r,C}^\star(\delta,c_1,u,v)$ in \eqref{trr} is bounded by
$$
\ll \frac{\min\{B^{-1},B^{k-\frac{3}{2}}\}q^{\varepsilon}}{q^{1+\frac{r}{2}}c_1\delta \sqrt{uv}}|S_r(0,-\overline{c_1}q^r\delta;q^2)|\ll \frac{q^{\varepsilon}}{c_1\delta \sqrt{uv}}.
$$
(Here we are using the trivial bound $g_{\delta,u,v}^\star(0)\ll \delta uv$.) Consequently the contribution of this part to $\mathcal T_{r,C}^\star$ in \eqref{tr'} is bounded by
\begin{align}
\label{lastone}
\ll \frac{N^2}{q^{5+\frac{1}{2}}}\sum_{\substack{\delta\leq Q\\q\nmid \delta}}
\sum_{\substack{c_1\leq Q\\c_1|(2\delta)^{\infty}}}(\delta,c_1)\mathop{\sum\sum}_{\substack{1\leq u,v\ll Q\\u,v|(2\delta)^{\infty}\\(u,v)=1}}\frac{q^{\varepsilon}}{c_1\delta \sqrt{uv}}\ll \frac{N^2}{q^{5+\frac{1}{2}}}\sum_{\delta\leq Q}\frac{q^{\varepsilon}}{\delta}\sum_{\substack{c_1\leq Q\\c_1|(2\delta)^{\infty}}}1\ll \frac{N^2}{q^{\frac{11}{2}}}q^{\varepsilon}.
\end{align}
This accounts for the second term in the bound given in \eqref{bdd2}, or the second term in the bound given in Theorem \ref{thm2}.

%========================================================================================================================
%========================================================================================================================

\section{Separation of variable and large sieve}
\label{sec-large}

Using the integral representation Lemma \ref{all-about-int}, we will now analyse the contribution of the positive frequencies, i.e. $c_2>0$, to $T_{r,C}^\star(\delta,c_1,u,v)$ in \eqref{trr}. To this end we need to get bounds for 
\begin{align}
\label{sumsum}
\frac{q^{\frac{r}{2}}\sqrt{B}}{Nc_1\delta uv}&\mathop{\iiint}_{(\mathbf{\sigma})}\frac{\tilde h(s_1)\tilde h(s_2)\tilde G(s_3)}{u^{\frac{1}{2}+s_2+s_3}v^{\frac{1}{2}+s_1+s_3}}\mathcal B(\mathbf{s})d\mathbf{s},
\end{align}
where 
\begin{align}
\label{bsum}
\mathcal B(\mathbf{s})=\sum_{c_2=1}^{\infty}& g_{\delta,u,v}^\star(c_2) S_r(c_2)\mathcal F_{r,c_1}(c_2;\delta,B;\mathbf{s})\mathop{\sum\sum}_{\substack{n,m=1\\(n,m)=1\\(nm,2\delta q)=1}}^{\infty}\frac{b_{\delta,\bar\chi}(n,c_2)}{n^{\frac{1}{2}+s_2+s_3}}\frac{b_{\delta,\chi}(m,c_2)}{m^{\frac{1}{2}+s_1+s_3}}.
\end{align}
Here we are using the shorthand notation $S_r(c_2)=S_r(\bar 2c_2,-\overline{c_1}q^r\delta;q^2)$, for which we will use the bound $|S_r(c_2)|\ll q(q,c_2)^{\frac{r}{2}}$. To ensure absolute convergence in the inner sums, a priori we put the restrictions that $\sigma_1+\sigma_3>1$ and $\sigma_2+\sigma_3>1$.  Also notice that we have extended the sums over $n$ and $m$ to all odd integers, for this manoeuvre we need to introduce an extra character modulo $4$, which we are going to ignore. Also we need to replace the Gauss sum $g(n,c_2)$ (also $g(m,c_2)$), which appears in the coefficients $b_{\delta,\psi}(n,c_2)$, by the multiplicative function 
$$
G_{c_2}(n)=\left(\frac{1-i}{2}+\left(\frac{-1}{n}\right)\frac{1+i}{2}\right)g(n,c_2).
$$
To separate the sums over $n$ and $m$ in \eqref{bsum}, we use Mobius inversion to get that $\mathcal B(\mathbf{s})$ is given by
\begin{align*}
\mathcal B(\mathbf{s})=\sum_{\substack{\theta=1\\(\theta,2\delta q)=1}}^{\infty}\mu(\theta)&\sum_{c_2=1}^\infty g_{\delta,u,v}^\star(c_2) S_r(c_2)\mathcal F_{r,c_1}(c_2;\delta,B;\mathbf{s})\mathop{\sum\sum}_{\substack{n,m=1\\(nm,2\delta q)=1}}^{\infty}\frac{b_{\delta,\bar\chi}(\theta n,c_2)}{(\theta n)^{\frac{1}{2}+s_2+s_3}}\frac{b_{\delta,\chi}(\theta m,c_2)}{(\theta m)^{\frac{1}{2}+s_1+s_3}}.\\
\end{align*}

Now we consider the $L$-series given by
$$
L_{\theta,\delta,\chi}(s;c_2)=\mathop{\sum}_{\substack{n=1\\(n,2\delta q)=1}}^{\infty}\frac{b_{\delta,\chi}(\theta n,c_2)}{(\theta n)^{\frac{1}{2}+s}}=\mathop{\sum}_{\substack{n=1\\(n,2\delta q)=1}}^{\infty}\frac{\chi(\theta n)G_{c_2}(\theta n)}{(\theta n)^{\frac{1}{2}+s}}\left(\frac{n\theta}{\delta}\right).
$$
The series converges absolutely for $\sigma$ sufficiently large, where it is also given by the Euler product
$$
L_{\theta,\delta,\chi}(s;c_2)=\prod_{p|\theta}\mathcal L_p'(s)\prod_{p\nmid 2\delta\theta q}\mathcal L_p(s),
$$
where 
$$
\mathcal L_p'(s)=\mathop{\sum}_{j=0}^{\infty}\frac{\chi(p^{j+1})G_{c_2}(p^{j+1})}{p^{(j+1)(\frac{1}{2}+s)}}\left(\frac{\delta}{p^{j+1}}\right),\;\;\;\text{and}\;\;\;
\mathcal L_p(s)=\mathop{\sum}_{j=0}^{\infty}\frac{\chi(p^{j})G_{c_2}(p^{j})}{p^{j(\frac{1}{2}+s)}}\left(\frac{\delta}{p^j}\right).
$$
In particular if $p\nmid 2\delta\theta qc_2$, we have
$$
\mathcal L_p(s)=\mathop{\sum}_{j=0}^{\infty}\frac{\chi(p^{j})G_{c_2}(p^{j})}{p^{j(\frac{1}{2}+s)}}=1+\frac{\chi(p)}{p^{s}}\left(\frac{\delta c_2}{p}\right).
$$
We write $c_2=c_{21}c_{22}^2$ and $\delta=\delta_1\delta_2^2$ with $c_{21}$ and $\delta_1$ square-free. Then it follows that we have a factorization
$$
L_{\theta,\delta,\chi}(s;c_2)=L\left(s,\chi\left(\frac{\delta_1c_{21}}{.}\right)\right)\tilde L_{\theta,\delta,\chi}(s;c_2),
$$
where the $L$-series $\tilde L_{\theta,\delta,\chi}(s;c_2)$ converges absolutely in the region $\sigma>\frac{1}{2}+\varepsilon$. Moreover in this domain we have
$$
\tilde L_{\theta,\delta,\chi}(s;c_2)\ll_{\varepsilon}\frac{(q\delta c_2)^{\varepsilon}}{\theta^{\frac{1}{2}+\varepsilon}}.
$$\\

Using the $L$-series we can write
\begin{align*}
\mathcal B(\mathbf{s})=\sum_{\substack{\theta=1\\(\theta,2\delta q)=1}}^{\infty}\mu(\theta)&\sum_{c_2=1}^\infty g_{\delta,u,v}^\star(c_2) S_r(c_2)\mathcal F_{r,c_1}(c_2;\delta,B;\mathbf{s})L_{\theta,\delta,\bar\chi}(s_2+s_3;c_2)L_{\theta,\delta,\chi}(s_1+s_3;c_2).
\end{align*}
We move the contours to $\sigma_1=\sigma_2=1$ and $\sigma_3=-\frac{1}{2}+\varepsilon$. Then applying Cauchy on the sum over $c_2$,  we are led to consider 
\begin{align}
\label{tc}
\frac{q^{\frac{r}{2}}\sqrt{B}}{Nc_1\delta (uv)^2}&\mathop{\iiint}_{(\mathbf{\sigma})}|\tilde h(s_1)\tilde h(s_2)\tilde G(s_3)|\mathcal B^\star_X(\mathbf{s})|d\mathbf{s}|,
\end{align}
where $\sigma$ is as above, and 
\begin{align*}
\mathcal B^\star_{X}(\mathbf{s})=\sum_{\theta=1}^{\infty}\frac{1}{\theta^{1+\varepsilon}}\sum_{c_{22}=1}^{\infty}\sum_{c_{21}\sim X}\left|g_{\delta,u,v}^\star(c_2) S_r(c_2)\mathcal F_{r,c_1}(c_2;\delta,B;\mathbf{s})\right|\left|L\left(s_1+s_3,\chi\left(\tfrac{\delta_1c_{21}}{.}\right)\right)\right|^2.
\end{align*}
Using approximate functional equation we can express the Dirichlet $L$-function as rapidly converging series with effective length given by the square-root of the analytic conductor. The analytic conductor is given by $[q^3,\frac{c_{21}\delta_1}{(c_{21},\delta_1)^2}]\left(3+|t_1+t_3|\right)$. Observe that if $q|c_{21}$ then the conductor drops, and in this case we have a better bound (in fact, with an extra saving of $\sqrt{q}$) compared to the generic case. In the generic case we note that (using the main result of \cite{HB})
\begin{align*}
\sum_{c_{22}\ll Q}(q\delta_1\delta_2,c_{22})\sum_{\substack{c_{21}\sim X/c_{22}^2\\q\nmid c_{21}}}(\delta_1,c_{21})\left|L\left(s_1+s_3,\chi\left(\tfrac{\delta_1 c_{21}}{.}\right)\right)\right|^2\ll (qT_{13})^{\varepsilon}T_{13}\left(X+\sqrt{q^3\delta_1X}\right),
\end{align*}
where $T_{13}=\left(3+|t_1+t_3|\right)$. \\

Substituting this bound in \eqref{tc}, using the bounds from Lemma \ref{bound-for-g}, $|S_r(c_2)|\ll q(q,c_2)^{\frac{r}{2}}$ (which is the Weil bound for Kloosterman sums) and \eqref{bound-for-F}, we obtain 
\begin{align}
\label{tc2}
\mathcal B_X^\star(\mathbf{s})\ll uv\delta_2q\min\{B^{-\frac{3}{2}},B^{k-1}\}T_{13}\left(X+\sqrt{q^3\delta_1 X}\right)(qT_{13})^{\varepsilon}.
\end{align}
According to Lemma \ref{all-about-int} in the worst case scenario $X=\frac{c_1q^{4+\varepsilon}}{\delta q^r}(B+B^{-1})$. Using \eqref{tc}, it follows that the contribution of the positive frequencies $c_2>0$ to \eqref{trr} is dominated by
$$
\frac{q^{1+\frac{r}{2}}}{Nc_1\delta_1\delta_2 uv}\min\{B^{-1},B^{k-\frac{1}{2}}\}\max\{B,B^{-1}\}\left(\frac{c_1q^{4+\varepsilon}}{\delta q^r}+\sqrt{\frac{c_1q^{7+\varepsilon}}{\delta_2^2q^r}}\right)q^{\varepsilon}.
$$
Now we observe that $\min\{B^{-1},B^{k-\frac{1}{2}}\}\max\{B,B^{-1}\}\ll 1$ (as $k\geq 2$). Bounding the contribution of the negative frequencies $c_2<0$ in exactly the same manner we obtain
\begin{align}
\label{trr-final-bd}
T_{r,C}^\star(\delta,c_1,u,v)\ll \frac{q^{\varepsilon}}{c_1\delta \sqrt{uv}}+\frac{q^{1+\frac{r}{2}}}{Nc_1\delta_1\delta_2 uv}\left(\frac{c_1q^{4+\varepsilon}}{\delta q^r}+\sqrt{\frac{c_1q^{7+\varepsilon}}{\delta_2^2q^r}}\right)q^{\varepsilon},
\end{align}
where the first term on the right hand side is the diagonal contribution which we obtained in Section~\ref{zero}. Substituting in \eqref{tr'} it follows that
$$
\mathcal T_{r,C}^\star\ll \frac{N^2}{q^{\frac{11}{2}}}q^{\varepsilon}+Nq^{\varepsilon}\sum_{\delta\leq Q}\frac{1}{\delta_1\delta_2}
\sum_{\substack{c_1\leq Q\\c_1|(2\delta)^{\infty}}}(\delta_2^2,c_1)\max\left\{\frac{1}{\delta_1\delta_2^2},\frac{1}{\delta_2\sqrt{qc_1}}\right\}\ll N\left(1+\frac{N}{q^{\frac{11}{2}}}\right)(Nq)^{\varepsilon}.
$$
(Recall that $\delta=\delta_1\delta_2^2$ with $\delta_1$ square-free.) This concludes the proof of the bound \eqref{bdd2}. As a consequence Theorem \ref{thm2} follows.

%===============================================================================================
%===============================================================================================

%========================================================================================================================
%========================================================================================================================

\end{document}